\documentclass{amsart}
\usepackage{setspace}
\usepackage{a4}
\usepackage{amssymb,amsmath,amsthm,latexsym}
\usepackage{amsfonts}
\usepackage{amsfonts}
\usepackage{graphicx}
\usepackage{textcomp}
\usepackage{cite}
\newtheorem{theorem}{Theorem}[section]

\newtheorem{conjecture}[theorem]{Conjecture}
\newtheorem{corollary}[theorem] {Corollary}
\newtheorem{definition}[theorem]{Definition}

\newtheorem{lemma} [theorem]{Lemma}

\newtheorem{remark}[theorem]{Remark}
\newtheorem{question}[theorem]{Question}
\title{This is the title}
\usepackage{amssymb}
\usepackage{amssymb}
\usepackage{amssymb}
\usepackage{amssymb}
\usepackage{amsmath}
\usepackage{tikz}
\usepackage{hyperref}
\usepackage{enumerate}
\usepackage{mathtools}
\usepackage{amsmath}
\usepackage{tikz}
\usepackage{fancyhdr}
\begin{document}

\begin{center}
{\bf{C*-ALGEBRAIC SCHUR PRODUCT THEOREM,   P\'{O}LYA-SZEG\H{O}-RUDIN QUESTION AND NOVAK'S CONJECTURE}}\\
K. MAHESH KRISHNA\\
 Department of Humanities and Basic  Sciences \\
Aditya College of Engineering and Technology \\
 Surampalem, East-Godavari\\
Andhra Pradesh 533 437 India\\
Email: kmaheshak@gmail.com \\
 \today
\end{center}

\hrule
\vspace{0.5cm}

\textbf{Abstract}: Striking result of Vyb\'{\i}ral  [\textit{Adv. Math.} 2020] says that  Schur product of positive matrices is   bounded below by the size of the matrix and the row sums of Schur product. Vyb\'{\i}ral used this result to prove the Novak's conjecture. In this paper, we define Schur product of matrices over arbitrary  C*-algebras and derive the results of Schur and   Vyb\'{\i}ral. As an application, we state  C*-algebraic version of Novak's conjecture and solve it for commutative unital C*-algebras. We formulate P\'{o}lya-Szeg\H{o}-Rudin question for the C*-algebraic Schur product of positive matrices.

\textbf{Keywords}: Schur/Hadamard product, Positive matrix, Hilbert C*-module, C*-algebra, Schur product theorem,  P\'{o}lya-Szeg\H{o}-Rudin question, Novak's conjecture.

\textbf{Mathematics Subject Classification (2020)}: 15B48, 46L05, 46L08. 

\tableofcontents

\section{Introduction}
Given matrices  $A\coloneqq [a_{j,k}]_{1\leq j,k\leq n}$ and $B\coloneqq [b_{j,k}]_{1\leq j,k\leq n}$ in the matrix ring $M_n(\mathbb{K})$, where $\mathbb{K}=\mathbb{R}$ or $\mathbb{C}$, the \textbf{Schur/Hadamard/pointwise product} of $A$ and $B$  is defined as 
\begin{align}\label{CLASSICAL}
A\circ B \coloneqq \begin{bmatrix} a_{j,k}b_{j,k}
\end{bmatrix}_{1\leq j,k\leq n}.
\end{align}
Recall that a matrix $A \in M_n(\mathbb{K})$ is said to be positive (also known as  self-adjont positive semidefinite) if it is self-adjoint and 
\begin{align*}
\langle Ax, x \rangle \geq 0, \quad \forall x \in \mathbb{K}^n,
\end{align*} 
where $\langle \cdot, \cdot \rangle $ is the standard Hermitian inner product (which is left linear right conjugate linear) on $\mathbb{K}^n$ (to move with the tradition of `operator algebra', by `positive' we only consider self-adjoint matrices).  In this case we write $ A\succeq 0$  and we write $ A\succeq B$ if all of  $A$, $B$  and $A-B$ are positive.  It is a century old result  that whenever $A,B \in M_n(\mathbb{K})$ are positive,  then their Schur product $A\circ B$ is positive. Schur  originally proved this result  in his famous `Crelle' paper \cite{SCHUR} and  today there are  varieties of proofs of this theorem. For a comprehensive look on Hadamard products we refer the reader to \cite{STYAN, OPPENHEIM, ZHANGBOOK, HORNHADAMARD, HORNJOHNSON1, HORNJOHNSON2}.

Once we know that the Schur product of two positive matrices is positive, then next step is to ask for a lower bound for the product, if exists. There are  series of papers obtaining  lower bounds for Schur product of positive correlation matrices \cite{LIUTRENKLER, ZHANG}, positive invertible matrices \cite{ANDO, BAPATKWONG, JOHNSONIN, FIEDLER, FIEDLERMARKHAM, VISICK, LIU, WANGZHANG} but for arbitrary positive matrices there are a couple of   recent results by  Vyb\'{\i}ral \cite{VYBIRAL} which we mention now. To state the results we need some notations.  Given a matrix $ M \in M_n(\mathbb{K})$, by $\overline{M}$ we mean the matrix obtained by taking conjugate of each entry of $M$. Conjugate transpose of a matrix $M$  is denoted by $M^*$ and $M^\text{T}$  denotes its  transpose. Notation $\text{diag}(M)$ denotes the  vector consisting of  the diagonal of matrix in the increasing subscripts. Matrix $E_n$ denotes the $n$ by $n$ matrix in $M_n(\mathbb{K})$ with all one's. Given a vector $x\in \mathbb{K}^n$, by $\text{diag}(x)$ we mean the $n$ by $n$ diagonal matrix obtained by putting $i$'th co-ordinate of $x$ as $(i,i)$ entry.
\begin{theorem}\cite{VYBIRAL}
	Let $ A \in M_n(\mathbb{K})$ be a positive matrix. Let $M=AA^*$ and $y\in \mathbb{K}^n$ be the vector of row sums of $A$. Then 
	\begin{align*}
	M\succeq \frac{1}{n}yy^*.
	\end{align*}
\end{theorem} 
\begin{theorem}\cite{VYBIRAL}\label{VIBIRALSECOND}
Let $ M,N \in M_n(\mathbb{K})$ be  positive matrices. Let $M=AA^*$, $M=BB^*$ and $y\in \mathbb{K}^n$ be the vector of row sums of $A\circ B$. Then 
\begin{align*}
M\circ N \succeq  (A\circ B)(A\circ B)^*\succeq\frac{1}{n}yy^*.
\end{align*}	
\end{theorem}
Immediate consequences of Theorem \ref{VIBIRALSECOND}  are the following.
\begin{corollary}\cite{VYBIRAL}
Let $ M\in M_n(\mathbb{K})$ be a 
 positive matrix.  Then 
\begin{align*}
M\circ \overline{M} \succeq \frac{1}{n}(\text{diag }M)(\text{diag  }M)^\text{T}
\end{align*}	
and 
\begin{align*}
M\circ M \succeq \frac{1}{n}(\text{diag }M)(\text{diag }M)^*.
\end{align*}	
\end{corollary}
\begin{corollary}\cite{VYBIRAL}\label{VYBIRALCOROLLARY}
Let $ M\in M_n(\mathbb{R})$ be a 
 positive matrix such that all diagonal entries are one's.  Then 
\begin{align*}
M\circ M \succeq \frac{1}{n}E_n.
\end{align*}	
\end{corollary}
Vyb\'{\i}ral used Corollary \ref{VYBIRALCOROLLARY} to solve two decades old Novak's conjecture which states as follows. 
\begin{theorem}\cite{NOVAKWO, NOVAK, HINRICHAVYBIRAL}\label{NOVAKCONJECTURE} (\textbf{Novak's conjecture}) 
	 The matrix 
	\begin{align*}
	\begin{bmatrix}
	\prod_{l=1}^{d}\frac{1+\cos (x_{j,l}-x_{k,l})}{2}-\frac{1}{n}
	\end{bmatrix}
	_{1\leq j,k \leq n}
	\end{align*}	
	is positive for all $n,d\geq 2$ and all choices of $x_j=(x_{j,1}, \dots, x_{j,d})\in \mathbb{R}^d$,  $\forall 1\leq j \leq n$.
\end{theorem}

 Theorem \ref{VIBIRALSECOND} is also used in the study of random variables, numerical integration, trigonometric polynomials and tensor product problems, see \cite{VYBIRAL, HINRICHSKRIEGNOVAKVYBIRAL}.

The purpose of this paper is to introduce the Schur product of  matrices over C*-algebras, obtain some fundamental results and to state some problems. A very handy tool which we use is the theory of Hilbert C*-modules. This was first introduced by Kaplansky \cite{KAPLANSKY} for commutative C*-algebras and later   by Paschke \cite{PASCHKE} and Rieffel \cite{RIEFFEL} for non commutative C*-algebras. The theory  attained a greater height from  the work of   Kasparov \cite{KASPAROV, BLACKADAR, JENSEN}. For an introduction to the subject  Hilbert C*-modules we refer \cite{LANCE, MANUILOVTROITSKY}. 
\begin{definition}\cite{KAPLANSKY, PASCHKE, RIEFFEL}
Let $\mathcal{A}$ be a  C*-algebra. A left module 	 $\mathcal{E}$  over $\mathcal{A}$ is said to be a (left) \textbf{Hilbert C*-module} if there exists a map $ \langle \cdot, \cdot \rangle: \mathcal{E}\times \mathcal{E} \to \mathcal{A}$ such that the following hold. 
\begin{enumerate}[\upshape(i)]
	\item $\langle x, x \rangle  \geq 0$, $\forall x \in \mathcal{E}$. If $x \in  \mathcal{E}$ satisfies $\langle x, x \rangle  =0 $, then $x=0$.
	\item $\langle x+y, z \rangle  =\langle x, z \rangle+\langle y, z \rangle$, $\forall x,y,z \in \mathcal{E}$.
	\item  $\langle ax, y \rangle  =a\langle x, y \rangle$, $\forall x,y \in \mathcal{E}$, $\forall a \in \mathcal{A}$.
	\item $\langle x, y \rangle=\langle y,x \rangle^*$, $\forall x,y \in \mathcal{E}$.
	\item $\mathcal{E}$ is complete w.r.t. the norm $\|x\|\coloneqq \sqrt{\|\langle x, x \rangle\|}$,  $\forall x \in \mathcal{E}$.
\end{enumerate}
\end{definition}
We are going to use the following inequality.
\begin{lemma}\cite{PASCHKE}\label{CAUCHYSCHWARZ} (Cauchy-Schwarz inequality for Hilbert C*-modules) If $\mathcal{E}$ is a Hilbert C*-module  over $\mathcal{A}$, then 
	\begin{align*}
	\langle x, y \rangle \langle y,x \rangle\leq \|\langle y, y \rangle\|\langle x, x \rangle, \quad \forall x,y \in \mathcal{E}.
	\end{align*}
	
\end{lemma}
We encounter the following \textbf{standard Hilbert C*-module} in this paper. Let $\mathcal{A}$ be a  C*-algebra and $\mathcal{A}^n$ be the left module over $\mathcal{A}$ w.r.t. natural operations. Modular $\mathcal{A}$-inner product on $\mathcal{A}^n$ is defined as 
\begin{align*}
\langle (x_j)_{j=1}^n, (y_j)_{j=1}^n\rangle \coloneqq \sum_{j=1}^{n}x_jy_j^*,\quad \forall (x_j)_{j=1}^n, (y_j)_{j=1}^n \in \mathcal{A}^n.
\end{align*}
 Hence the norm on $\mathcal{A}^n$ becomes 
  \begin{align*}
 \|(x_j)_{j=1}^n\|\coloneqq \left\|\sum_{j=1}^{n}x_jx_j^*\right\|^\frac{1}{2}, \quad \forall (x_j)_{j=1}^n \in \mathcal{A}^n.
 \end{align*}
This paper is organized as follows. In Section \ref{SECTION2} we define Schur/Hadamard/pointwise product of two matrices over C*-algebras (Definition \ref{CALGEBRAICDEF}). This is not a direct mimic of Schur product of matrices over scalars. After the definition of Schur product, we derive Schur product theorem for matrices over commutative C*-algebras (Theorem \ref{COMMUTATIVESCHURPRODUCT}), $\sigma$-finite W*-algebras or AW*-algebras (Theorem \ref{NONCOMMUTATIVESCHURPRODUCT}).  
Followed by these results, we ask P\'{o}lya-Szeg\H{o}-Rudin question for positive matrices over C*-algebras (Question \ref{POLYASZEGO}). We then  develop the paper following the developments by  Vyb\'{\i}ral in \cite{VYBIRAL} to the setting of C*-algebras. In Section \ref{SECTION3} we first derive lower bound for positive matrices over C*-algebras (Theorem \ref{FIRSTTHEOREM}) and using that we derive  lower bounds for Schur product (Theorem \ref{MYTHEOREM} and  Corollaries \ref{FIRSTCOROLLARY}, \ref{SECONDCOROLLARY}).  We later  state  C*-algebraic version of Novak's conjecture (Conjecture \ref{CALGEBRAICNOVAKNON}). We solve it for commutative unital C*-algebras (Theorem \ref{CALGEBRAICNOVAKCONJECTURE}). Finally we end the paper by asking  Question \ref{LASTQUESTION}.

\section{C*-algebraic Schur product, Schur product theorem and  P\'{o}lya-Szeg\H{o}-Rudin question}\label{SECTION2}
We first recall the basics in the theory  of matrices over C*-algebras. More information can be found in   \cite{WEGGEOLSEN,MURPHY}. 
Let  $\mathcal{A}$ be a unital C*-algebra and  $n$ be a natural number. Set  $M_n(\mathcal{A})$ is defined as the set of all $n$ by $n$ matrices over $\mathcal{A}$ which becomes  an algebra with respect to natural matrix operations. The involution of an element  $A\coloneqq [a_{j,k}]_{1\leq j,k \leq n}\in M_n(\mathcal{A})$ as $A^*\coloneqq [a_{k,j}^*]_{1\leq j,k \leq n}$. Then $M_n(\mathcal{A})$ becomes a *-algebra. \textbf{Gelfand-Naimark-Segal theorem} says that there exists a unique universal representation $(\mathcal{H}, \pi) $, where $\mathcal{H}$ is a Hilbert space, $\pi:M_n(\mathcal{A})\to M_n(\mathcal{B}(\mathcal{H}))$ is an isometric *-homomorphism. Using this, the norm on $M_n(\mathcal{A})$ is  defined as 
\begin{align*}
\|A\|\coloneqq \|\pi(A)\|, \quad \forall A \in M_n(\mathcal{A})
\end{align*}
which  makes $M_n(\mathcal{A})$ as a C*-algebra (where $\mathcal{B}(\mathcal{H})$ is the C*-algebra of all continuous  linear operators on $\mathcal{H}$ equipped with the operator-norm). \\
We define C*-algebraic Schur product as follows.
\begin{definition}\label{CSCHUR}\label{CALGEBRAICDEF}
Let $\mathcal{A}$ be a  C*-algebra. Given $A\coloneqq [a_{j,k}]_{1\leq j,k \leq n}, B\coloneqq [b_{j,k}]_{1\leq j,k \leq n}\in M_n(\mathcal{A})$, we define the  \textbf{ C*-algebraic Schur/Hadamard/pointwise product} of $A$ and $B$ as	
\begin{align}\label{MYDEFINITION}
A \circ B \coloneqq \frac{1}{2}\begin{bmatrix}
a_{j,k}b_{j,k}+b_{j,k}a_{j,k}
\end{bmatrix}_{1\leq j,k \leq n}.
\end{align}	
\end{definition}
Whenever the C*-algebra is commutative, then (\ref{MYDEFINITION}) becomes
\begin{align*}
A \circ B =\begin{bmatrix} 
a_{j,k}b_{j,k}\end{bmatrix}_{1\leq j,k \leq n}.
\end{align*}
In particular, Definition \ref{CSCHUR} reduces to the definition of classical Schur product given in Equation (\ref{CLASSICAL}). From a direct computation, we have the following result.
\begin{theorem}\label{SCHURPRODUCTPROPERTIES}
	Let $\mathcal{A}$ be a unital C*-algebra and let  $ A, B, C \in M_n(\mathcal{A})$. Then 
\begin{enumerate}[\upshape(i)]
	\item $A \circ B=B \circ A$.
	\item $(A \circ B)^*=A^* \circ B^*$. In particular, if $A$ and $B$ are self-adjoint, then $A\circ B$ is self-adjoint.
	\item $A\circ (B+C)=A \circ B+A \circ C$.
	\item $(A + B)\circ C=A \circ C+B \circ C$.
\end{enumerate}	
\end{theorem}
One of the most important difference of Definition \ref{CALGEBRAICDEF} from the classical Schur product is that the product may not be associative, i.e.,  $(A \circ B)\circ C\neq A \circ (B \circ C)$ in general.\\
Similar to the scalar case,  $A\coloneqq [a_{j,k}]_{1\leq j,k \leq n}\in M_n(\mathcal{A})$ is said to be positive if it is self-adjoint and 
\begin{align*}
\langle Ax, x \rangle \geq 0, \quad \forall x \in \mathcal{A}^n, 
\end{align*} 
where  $ \geq$ is the partial order on the set of all positive elements of $\mathcal{A}$. In this case we write $A\succeq 0$.  It is  well-known  in the theory of C*-algebras that the set of all positive elements in a C*-algebra is a closed  positive cone.  We then have that  the set of all positive matrices in $M_n(\mathcal{A})$ is a closed  positive cone. Here comes the first version of C*-algebraic Schur product theorem.

\begin{theorem}\label{COMMUTATIVESCHURPRODUCT} (\textbf{Commutative C*-algebraic version of Schur product theorem})
	Let $\mathcal{A}$ be a  commutative unital C*-algebra. If  $M, N\in M_n(\mathcal{A})$ are positive, then their Schur product $M\circ N$ is also positive. 
\end{theorem}
\begin{proof}
	Let $ x \in \mathcal{A}^n$ and define $L\coloneqq (M^\frac{1}{2})^\text{T}(\text{diag } x) (N^\frac{1}{2})^\text{T}$. First note that $M\circ N$ is self-adjoint. Using the commutativity of  C*-algebra, we get 
	\begin{align*}
	\langle (M \circ N ) x, x\rangle &=x^*(M \circ N )x=\text{Tr}((\text{diag }x^*)M(\text{diag } x) N^\text{T})\\
	&=\text{Tr}((\text{diag }x^*)M(\text{diag } x) (N^\frac{1}{2})^\text{T}(N^\frac{1}{2})^\text{T})\\
	&=\text{Tr}((N^\frac{1}{2})^\text{T}(\text{diag }x^*)M(\text{diag } x) (N^\frac{1}{2})^\text{T})\\
	&=\text{Tr}((N^\frac{1}{2})^\text{T}(\text{diag }x^*)(M^\frac{1}{2})^\text{T}(M^\frac{1}{2})^\text{T}(\text{diag } x) (N^\frac{1}{2})^\text{T})\\
	&=\text{Tr}(L^*L)\geq 0.
	\end{align*}
	Since $x$ was arbitrary, the result follows.
\end{proof}
 In the sequel, we use the following notation. Given $M\in M_n(\mathcal{A})$, we define
\begin{align*}
(M^\circ)^n\coloneqq M\circ \cdots \circ M  \quad (\text{n times}), \forall n\geq 1,  \quad (M^\circ)^0\coloneqq I \quad \text{(identity matrix in } M_n(\mathcal{A})).
\end{align*} 
\begin{corollary}\label{POLYNOMIALPOSITIVE}
Let $\mathcal{A}$ be a  commutative unital C*-algebra. Let  $M\in M_n(\mathcal{A})$ be positive. If $a_0+a_1z+a_2z^2+\cdots +a_nz^n$ is any polynomial with coefficients from  $\mathcal{A}$ with all $a_0,\dots, a_n$ are positive elements of $\mathcal{A}$, then the matrix 
\begin{align*}
a_0I+a_1M+a_2(M^\circ)^2+\cdots +a_n(M^\circ)^n \in M_n(\mathcal{A})
\end{align*}
is positive.
\end{corollary}
\begin{proof}
This follows from 	Theorem \ref{COMMUTATIVESCHURPRODUCT} and Mathematical induction.
\end{proof}
\begin{remark}
Note that we used commutativity of C*-algebra in the proof of Theorem \ref{COMMUTATIVESCHURPRODUCT} and thus it can not be carried over to non commutative C*-algebras.	
\end{remark}
 Theorem \ref{COMMUTATIVESCHURPRODUCT} leads us to seek a similar result for non commutative C*-algebras. At present we don't know Schur product theorem for positive matrices over arbitrary C*-algebras. For the purpose of definiteness, we state it as an open problem.
\begin{question}\label{SCHURPRODUCTPROBLEM}
	 Let $\mathcal{A}$ be a   C*-algebra. \textbf{Given positive matrices  $M, N\in M_n(\mathcal{A})$, does $M\circ N$ is  positive?} In other  words, \textbf{classify those   C*-algebras $\mathcal{A}$  such that   $M\circ N$ is    positive  whenever $M, N\in M_n(\mathcal{A})$ are positive.}
\end{question}
To make some progress to  Question \ref{SCHURPRODUCTPROBLEM}, we give an  affirmative answer  for certain classes of C*-algebras (von Neumann algebras).  To do so we need spectral theorem for matrices over C*-algebras. First let us recall two definitions. 
\begin{definition}\cite{MANUILOV}
A W*-algebra is called 	\textbf{$\sigma$-finite} if it contains no more than a countable set of mutually orthogonal projections.
\end{definition}
\begin{definition}\cite{KAPLANSKYANNALS}
	A C*-algebra $\mathcal{A}$ is called an \textbf{AW*-algebra} if the following conditions hold.
	\begin{enumerate}[\upshape(i)]
		\item Any set of orthogonal projections has  supremum.
		\item Any maximal commutative self-adjoint  subalgebra of  $\mathcal{A}$ is generated by its projections. 
	\end{enumerate}
\end{definition}
\begin{theorem} \cite{MANUILOV, HEUNENREYES}\label{SPECTRALTHEOREM}
	(\textbf{Spectral theorem for Hilbert C*-modules}) Let $\mathcal{A}$ be a $\sigma$-finite W*-algebra or an AW*-algebra. If  
	$M \in  M_n(\mathcal{A})$ is normal, then there exists a unitary matrix $U \in  M_n(\mathcal{A})$ such that $UMU^*$ is a diagonal matrix. 
\end{theorem}
\begin{theorem}\label{NONCOMMUTATIVESCHURPRODUCT}(\textbf{Non commutative C*-algebraic Schur product theorem})
Let $\mathcal{A}$ be a $\sigma$-finite W*-algebra or an AW*-algebra and  
$M, N\in M_n(\mathcal{A})$ be positive. Let $U=[u_{j,k}]_{1\leq j,k\leq n}, V=[v_{j,k}]_{1\leq j,k\leq n} \in  M_n(\mathcal{A})$ be unitary  such that 
\begin{align*}
M=U\begin{bmatrix}
\lambda_1& 0&  \cdots & 0\\
0& \lambda_2 & \cdots & 0\\
\vdots & \vdots & \cdots & \vdots \\
0& 0 & \cdots & \lambda_n\\
\end{bmatrix}
U^*, \quad N=V\begin{bmatrix}
\mu_1& 0&  \cdots & 0\\
0& \mu_2 & \cdots & 0\\
\vdots & \vdots & \cdots & \vdots \\
0& 0 & \cdots & \mu_n\\
\end{bmatrix}
V^*,
\end{align*}
for some $\lambda_1,\dots, \lambda_n , \mu_1,\dots, \mu_n \in \mathcal{A}$. If all $\lambda_j, \mu_k, u_{l,m}, v_{r,s}$, $1\leq j,k,l,m,r,s\leq n$ commute with each other,  then the Schur product $M\circ N$ is also positive.	
\end{theorem}
\begin{proof}
Let $\{u_1, \dots, u_n\}$ be columns of $U$  and $\{v_1, \dots, v_n\}$ be columns  of $V$. Then 
\begin{align*}
A =\sum_{j=1}^{n}\lambda_j u_j u_j^*, \quad  B =\sum_{k=1}^{n}\mu_k v_k v_k^*
\end{align*}
where $\lambda_1,\dots, \lambda_n $ are eigenvalues of $A$, $\{u_1, \dots, u_n\}$ is an orthonormal basis for $\mathcal{A}^n$,  $\mu_1,\dots, \mu_n $ are eigenvalues of $B$ and   $\{v_1, \dots, v_n\}$ is an orthonormal basis for $\mathcal{A}^n$ (they exist from Theorem \ref{SPECTRALTHEOREM}).	Definition \ref{MYDEFINITION} of Schur product says that $M\circ N$ is self-adjoint. It is well known in the theory of C*-algebras that sum of positive elements in a C*-algebra is positive and the product of two commuting positive elements is positive.  This observation, Theorem \ref{SCHURPRODUCTPROPERTIES} and the following calculation shows that $M\circ N$ is positive: 
\begin{align*}
M\circ N &=\left(\sum_{j=1}^{n}\lambda_j u_j u_j^*\right)\circ \left(\sum_{k=1}^{n}\mu_k v_k v_k^*\right)=\sum_{j=1}^{n}\sum_{k=1}^{n}\lambda_j \mu_k (u_ju_j^*)\circ (v_kv_k^*)\\
&=\sum_{j=1}^{n}\sum_{k=1}^{n}\lambda_j \mu_k (u_j \circ v_k)(u_j \circ v_k)^*\succeq 0.
\end{align*}
\end{proof}
Since the \textbf{spectral theorem fails for matrices over C*-algebras} (see \cite{GROVE, KADISON, KADISON2}), proof of Theorem \ref{NONCOMMUTATIVESCHURPRODUCT} can not be executed for  arbitrary   C*-algebras.\\
Given certain order structure, one naturally considers  functions (in a suitable way) which preserve the order. For matrices over C*-algebras, we formulate this in the following definition.
\begin{definition}\label{POLYASZEGO}
	Let $\mathcal{B}$ be a subset of  a C*-algebra $\mathcal{A}$ and  $n$ be a  natural number. Define   $\mathcal{P}_n(\mathcal{B})$ as the set of all $n$ by $n$ positive matrices with entries from $\mathcal{B}$. Given a function $f:\mathcal{B}\to \mathcal{A}$,
	define a function 
	\begin{align*}
	\mathcal{P}_n(\mathcal{B}) \ni A\coloneqq \begin{bmatrix}
	a_{j,k}
	\end{bmatrix}_{1\leq j,k\leq n} \mapsto f[A]\coloneqq \begin{bmatrix}
	f(a_{j,k})
	\end{bmatrix}_{1\leq j,k\leq n}\in M_n(\mathcal{A}).
	\end{align*}
	A function  $f:\mathcal{B}\to \mathcal{A}$ is said to be a \textbf{positivity preserver in all  dimensions} if 
	\begin{align*}
	f[A]\in \mathcal{P}_n(\mathcal{A}), \quad \forall A \in \mathcal{P}_n(\mathcal{B}), \quad \forall n \in \mathbb{N}.
	\end{align*}
	A	function  $f:\mathcal{B}\to \mathcal{A}$ is said to be a \textbf{positivity preserver in fixed  dimension} $n$ if 
	\begin{align*}
	f[A]\in \mathcal{P}_n(\mathcal{A}), \quad \forall A \in \mathcal{P}_n(\mathcal{B}).
	\end{align*}
\end{definition}
We now have the important C*-algebraic P\'{o}lya-Szeg\H{o}-Rudin open problem.
\begin{question}\label{QUESTIONPOLYA}
	(\textbf{P\'{o}lya-Szeg\H{o}-Rudin question for C*-algebraic Schur product of positive matrices}) 
	Let $\mathcal{B}$ be a subset of a (commutative)  C*-algebra $\mathcal{A}$ and $\mathcal{P}_n(\mathcal{B})$ be as in Definition \ref{POLYASZEGO}.
	\begin{enumerate}[\upshape(i)]
		\item \textbf{Characterize $f$ such that $f$ is a positivity preserver for all $n \in \mathbb{N}$.}
		\item \textbf{Characterize $f$ such that $f$ is a positivity preserver for fixed $n$.}
	\end{enumerate}
\end{question}

Answer to (i) in Question \ref{QUESTIONPOLYA} in the case $\mathcal{A}=\mathbb{R}$ (which is due to P\'{o}lya and Szeg\H{o} \cite{POLYASZEGO}) is known from the works of Schoenberg \cite{SCHOENBERGDUKE}, Vasudeva \cite{VASUDEVA}, Rudin \cite{RUDINDUKE}, Christensen and Ressel \cite{CHRISTENSENRESSEL}. Further the answer to Question (i) in the case $\mathcal{A}=\mathbb{C}$ (which is due to Rudin \cite{RUDINDUKE}) is also known from the work of Herz \cite{HERZ}. There are certain partial answers to  (ii) in Question \ref{QUESTIONPOLYA} from the works of Horn \cite{HORNTHE}, Belton, Guillot, Khare, Putinar, Rajaratnam and Tao \cite{BELTONGUILLOTKHAREPUTINAR, GUILLOTKHARERAJARATNAM, APOORVA1, APOORVA2, TAOAPOORVAKHARE}.\\
Corollary \ref{POLYNOMIALPOSITIVE} and the observation that the set of all positive matrices in  $M_n(\mathcal{A})$ is a  closed set   gives a partial answer to (i) in Question \ref{QUESTIONPOLYA}.

\begin{theorem}
Let $\mathcal{A}$ be a commutative unital C*-algebra. Let the power series $f(z)	\coloneqq \sum_{n=0}^{\infty}a_nz^n$ over $\mathcal{A}$ be convergent on a subset $\mathcal{B}$  of   $\mathcal{A}$. If all  $a_n$'s are positive elements of $\mathcal{A}$, then the matrix 
\begin{align*}
f[A]=\sum_{n=0}^{\infty}a_n(A^\circ)^n \in  M_m(\mathcal{A})
\end{align*}
is positive for all positive $ A \in M_m(\mathcal{A})$, for all $m \in \mathbb{N}$. In other words, a \textbf{convergent power series over a commutative unital C*-algebra with positive elements as coefficients is a positivity preserver in all dimensions}.
\end{theorem}

\section{Lower bounds for C*-algebraic Schur product}\label{SECTION3}
Our first result is on the lower bound of positive matrices over C*-algebras.
\begin{theorem}\label{FIRSTTHEOREM}
Let $\mathcal{A}$ be a unital C*-algebra (need not be commutative)	and  $ A \in M_n(\mathcal{A})$ be a positive matrix. Let $M=AA^*$ and $y\in \mathcal{A}^n$ be the vector of row sums of $A$. Then 
	\begin{align*}
	M\succeq \frac{1}{n}yy^*,
	\end{align*}
	i.e., 
	\begin{align}\label{FUNDAMENTAL}
	\langle Mx,x\rangle \geq \frac{1}{n}\langle x,x\rangle, \quad \forall x \in \mathcal{A}^n.
	\end{align}
\end{theorem}
\begin{proof}
Set

\begin{align*}
A\coloneqq \begin{bmatrix}
a_{1,1} & a_{1,2} &  \cdots & a_{1,n}\\
a_{2,1} & a_{2,2} &  \cdots& a_{2,n}\\
\vdots &\vdots  &   &\vdots\\
a_{n,1}&a_{n,2}&\cdots& a_{n,n}
\end{bmatrix} \in M_n(\mathcal{A}), \quad x\coloneqq 
\begin{pmatrix}
x_1 \\
x_2 \\
\vdots \\
x_n\\
\end{pmatrix}\in \mathcal{A}^n, \quad y\coloneqq 
\begin{pmatrix}
y_1 \\
y_2 \\
\vdots \\
y_n\\
\end{pmatrix}\in \mathcal{A}^n.
\end{align*}
Since $y$ is the vector of row sums of $A$, we have 
\begin{align*}
y_j=\sum_{k=1}^{n}a_{j,k}, \quad \forall 1\leq j \leq n.
\end{align*}
Consider 
\begin{align*}
\langle Mx, x\rangle &=\langle AA^*x, x\rangle=\langle A^*x, A^*x\rangle=\left\langle \begin{pmatrix}
\sum_{k=1}^{n}a_{k,1}^*x_k \\
\sum_{k=1}^{n}a_{k,2}^*x_k\\
\vdots \\
\sum_{k=1}^{n}a_{k,n}^*x_k\\
\end{pmatrix}, \begin{pmatrix}
\sum_{l=1}^{n}a_{l,1}^*x_l \\
\sum_{l=1}^{n}a_{l,2}^*x_l \\
\vdots \\
\sum_{l=1}^{n}a_{l,n}^*x_l\\
\end{pmatrix}\right\rangle \\
&=\sum_{j=1}^{n}\left(\sum_{k=1}^{n}a_{k,j}^*x_k\right)\left(\sum_{l=1}^{n}a_{l,j}^*x_l\right)^*=\sum_{j=1}^{n}\sum_{k=1}^{n}\sum_{l=1}^{n}a_{k,j}^*x_kx_l^*a_{l,j}
\end{align*}
which is the left side of Inequality (\ref{FUNDAMENTAL}). 
Set 
\begin{align*}
e_n\coloneqq 
\begin{pmatrix}
1 \\
1 \\
\vdots \\
1\\
\end{pmatrix}\in \mathcal{A}^n, \quad z\coloneqq  \begin{pmatrix}
z_1 \\
z_2 \\
\vdots \\
z_n\\
\end{pmatrix}\coloneqq 
\begin{pmatrix}
\sum_{k=1}^{n}a_{k,1}^*x_k \\
\sum_{k=1}^{n}a_{k,2}^*x_k \\
\vdots \\
\sum_{k=1}^{n}a_{k,n}^*x_k\\
\end{pmatrix}\in \mathcal{A}^n.
\end{align*}
We now consider the right side of Inequality (\ref{FUNDAMENTAL}) and use Lemma \ref{CAUCHYSCHWARZ} to get 

\begin{align*}
\frac{1}{n}\langle yy^*x, x\rangle &=\frac{1}{n}\langle y^*x, y^*x\rangle\\
&=\frac{1}{n}\left\langle \begin{pmatrix}
y_1^* &y_2^*&\dots &  y_n^*
\end{pmatrix}\begin{pmatrix}
x_1 \\
x_2 \\
\vdots \\
x_n\\
\end{pmatrix}, \begin{pmatrix}
y_1^* &y_2^*&\dots &  y_n^*
\end{pmatrix}\begin{pmatrix}
x_1 \\
x_2 \\
\vdots \\
x_n\\
\end{pmatrix}\right\rangle \\
&=\frac{1}{n}\left(\sum_{k=1}^{n}y_k^*x_k\right)\left(\sum_{l=1}^{n}y_l^*x_l\right)^*=\frac{1}{n}\left(\sum_{k=1}^{n}y_k^*x_k\right)\left(\sum_{l=1}^{n}x_l^*y_l\right)\\
&=\frac{1}{n}\left(\sum_{k=1}^{n}\sum_{r=1}^{n}a_{k,r}^*x_k\right)\left(\sum_{l=1}^{n}x_l^*\sum_{s=1}^{n}a_{l,s}\right)\\
&=\frac{1}{n}\left(\sum_{k=1}^{n}\sum_{r=1}^{n}a_{k,r}^*x_k\right)\left(\sum_{l=1}^{n}\sum_{s=1}^{n}x_l^*a_{l,s}\right)\\
&=\frac{1}{n}\left(\sum_{r=1}^{n}\sum_{k=1}^{n}a_{k,r}^*x_k\right)\left(\sum_{s=1}^{n}\sum_{l=1}^{n}x_l^*a_{l,s}\right)\\
&=\frac{1}{n}\left(\sum_{r=1}^{n}\left(\sum_{k=1}^{n}a_{k,r}^*x_k\right).1\right)\left(\sum_{s=1}^{n}1.\left(\sum_{l=1}^{n}a_{l,s}^*x_l\right)^*\right)\\
&=\frac{1}{n}\left\langle
\begin{pmatrix}
\sum_{k=1}^{n}a_{k,1}^*x_k \\
\sum_{k=1}^{n}a_{k,2}^*x_k \\
\vdots \\
\sum_{k=1}^{n}a_{k,n}^*x_k\\
\end{pmatrix}, \begin{pmatrix}
1 \\
1 \\
\vdots \\
1\\
\end{pmatrix} \right\rangle 
\left\langle\begin{pmatrix}
1 \\
1 \\
\vdots \\
1\\
\end{pmatrix}, 
\begin{pmatrix}
\sum_{k=1}^{n}a_{k,1}^*x_k \\
\sum_{k=1}^{n}a_{k,2}^*x_k \\
\vdots \\
\sum_{k=1}^{n}a_{k,n}^*x_k\\
\end{pmatrix}\right\rangle \\
&=\frac{1}{n}\langle z, e_n\rangle \langle  e_n, z\rangle \leq \frac{1}{n}\|\langle e_n,e_n \rangle \|\langle z, z\rangle =\langle z, z\rangle\\
&=\sum_{j=1}^nz_jz_j^*=\sum_{j=1}^n\left(\sum_{k=1}^{n}a_{k,j}^*x_k\right)\left(\sum_{l=1}^{n}a_{l,j}^*x_l\right)^*\\
&=\sum_{j=1}^n\sum_{k=1}^{n}\sum_{l=1}^{n}a_{k,j}^*x_kx_l^*a_{l,j}=\langle Mx, x\rangle 
\end{align*}
which is the required inequality.
\end{proof}

\begin{theorem}\label{MYTHEOREM}
Let $\mathcal{A}$ be a commutative unital C*-algebra.	Let $ M,N \in M_n(\mathcal{A})$ be  positive matrices. Let $M=AA^*$, $M=BB^*$ and $y\in \mathcal{A}^n$ be the vector of row sums of $A\circ B$. Then 
	\begin{align*}
	M\circ N \succeq  (A\circ B)(A\circ B)^*\succeq\frac{1}{n}yy^*.
	\end{align*}	
\end{theorem}
\begin{proof}
Let $\{A_1, \dots, A_n\}$ be columns of $A$  and $\{B_1, \dots, B_n\}$ be columns  of $B$. Then using commutativity and Theorem \ref{FIRSTTHEOREM} we get 

\begin{align*}
M\circ N & =(AA^*)\circ (BB^*)=\left(\sum_{j=1}^{n}A_jA_j^*\right)\circ \left(\sum_{k=1}^{n}B_kB_k^*\right)\\
&=\sum_{j=1}^{n}\sum_{k=1}^{n}((A_jA_j^*)\circ (B_kB_k^*))=\sum_{j=1}^{n}\sum_{k=1}^{n}(A_j\circ B_k)(A_j\circ B_k)^*\\
&\succeq\sum_{j=1}^{n} (A_j\circ B_j)(A_j\circ B_j)^*=(A\circ B)(A\circ B)^*\succeq\frac{1}{n}yy^*.
\end{align*}	
\end{proof}
\begin{corollary}\label{FIRSTCOROLLARY}
	Let $ M\in M_n(\mathcal{A})$ be a  positive matrix.  Then 
	\begin{align*}
	M\circ M \succeq \frac{1}{n}(\text{diag }M)(\text{diag }M)^*.
	\end{align*}	
\end{corollary}
\begin{proof}
	Let $B=A$ in Theorem \ref{MYTHEOREM}. Result follows by noting that diagonal entries of $M$ are row sums of $A\circ A$.  
\end{proof}
Following corollary is immediate from Corollary \ref{FIRSTCOROLLARY}.
\begin{corollary}\label{SECONDCOROLLARY}
	Let $ M\in M_n(\mathcal{A})$ be a  positive matrix such that all diagonal entries  of $ M$ are one's.  Then 
	\begin{align*}
	M\circ M \succeq \frac{1}{n}E_n.
	\end{align*}	
\end{corollary}

\section{C*-algebraic Novak's conjecture}

It is well known that the \textbf{exponential map}
\begin{align*}
e:\mathcal{A}\ni x \mapsto e^x \coloneqq \sum_{n=0}^{\infty}\frac{x^n}{n!}\in \mathcal{A}
\end{align*}
is a well defined map on a unital C*-algebra (more is true, it is well-defined on unital Banach algebras). Using this map and from the definition of trigonometric functions (for instance, see Chapter 8  in \cite{RUDINBOOK})  we define C*-algebraic sine and cosine functions as follows.
\begin{definition}
	Let $\mathcal{A}$ be a  unital C*-algebra. Define the \textbf{C*-algebraic sine} function by 
	\begin{align*}
		\sin: \mathcal{A}\ni x \mapsto \sin x \coloneqq \frac{e^{ix}-e^{-ix}}{2i}\in  \mathcal{A}.
	\end{align*}
Define the \textbf{C*-algebraic cosine} function by
\begin{align*}
	\cos: \mathcal{A}\ni x \mapsto \cos x \coloneqq \frac{e^{ix}+e^{-ix}}{2}\in  \mathcal{A}.
\end{align*}
\end{definition}
By a direct computation, we have the following result. The result also shows the similarity and differences of C*-algebraic trigonometric functions with  usual trigonometric functions.
\begin{theorem}\label{TRIGONOMETRICTHEOREM}
Let $\mathcal{A}$ be a 	unital C*-algebra. Then 
	\begin{enumerate}[\upshape (i)]
		\item $\sin(-x)=-\sin x, \forall x \in \mathcal{A}.$
		\item $\cos(-x)=\cos x, \forall x \in \mathcal{A}.$
		\item $\sin(x+y)=\sin x \cos y+ \cos x \sin y, \forall x,y  \in \mathcal{A}$ such that $xy=yx$.
		\item $\cos(x+y)=\cos x \cos y- \sin x \sin y, \forall x,y  \in \mathcal{A}$ such that $xy=yx$.
		\item $(\sin x)^*=\sin x^*, \forall x \in \mathcal{A}.$
		\item $(\cos x)^*=\cos x^*, \forall x \in \mathcal{A}.$
		\item $\sin^2 x+\cos^2 x=1, \forall x \in \mathcal{A}.$
	\end{enumerate}
\end{theorem}
In the sequel, by $\mathcal{A}_\text{sa}$ we mean the set of all self-adjoint elements in the unital C*-algebra $\mathcal{A}$.
Motivated from Novak's conjecture (Theorem \ref{NOVAKCONJECTURE}), we formulate the following conjecture. 
\begin{conjecture}\label{CALGEBRAICNOVAKNON}(\textbf{C*-algebraic Novak's conjecture})
Let $\mathcal{A}$ be a  unital C*-algebra. Then the matrix 
\begin{align*}
\begin{bmatrix}\prod_{l=1}^{d}
\frac{1+\cos (x_{j,l}-x_{k,l})}{2}-\frac{1}{n}
\end{bmatrix}_{1\leq j,k \leq n}
\end{align*}	
is positive for all $n,d\geq 2$ and all choices of $x_j=(x_{j,1}, \dots, x_{j,d})\in \mathcal{A}_\text{sa}^d$,  $\forall 1\leq j \leq n$.	
\end{conjecture}

We solve a special case of Conjecture \ref{CALGEBRAICNOVAKNON}.
\begin{theorem}\label{CALGEBRAICNOVAKCONJECTURE}(\textbf{Commutative C*-algebraic Novak's conjecture})
Let $\mathcal{A}$ be a commutative unital C*-algebra. Then the matrix 
\begin{align*}
 \begin{bmatrix}\prod_{l=1}^{d}
 \frac{1+\cos (x_{j,l}-x_{k,l})}{2}-\frac{1}{n}
 \end{bmatrix}_{1\leq j,k \leq n}
\end{align*}	
is positive for all $n,d\geq 2$ and all choices of $x_j=(x_{j,1}, \dots, x_{j,d})\in \mathcal{A}_\text{sa}^d$,  $\forall 1\leq j \leq n$.
\end{theorem}
\begin{proof}
 We first show that the matrix 
\begin{align*}
A\coloneqq \begin{bmatrix}
\cos (z_{j}-z_{k})\end{bmatrix}_{1\leq j,k \leq n}
\end{align*}	
is positive for all $n,d\geq 2$ and all choices of $z_1,\dots , z_n\in \mathcal{A}_{\text{sa}}$. First note that  Theorem \ref{TRIGONOMETRICTHEOREM}  says that the matrix $A$ is self adjoint. An important theorem used by Vyb\'{\i}ral in his proof of Novak's conjecture is the Bochner theorem \cite{REEDSIMON}. Since  \textbf{Bochner theorem} for C*-algebras is probably not known, we use Theorem \ref{TRIGONOMETRICTHEOREM} and  make a direct computation which is inspired from computation done in  	\cite{STEWART}.  Let  $y=(y_1,\dots, y_n)\in \mathcal{A}_{\text{sa}}^d$. Then 
\begin{align*}
\langle Ay, y\rangle &=\sum_{j=1}^{n}\sum_{k=1}^{n}(\cos (z_j-z_k))y_jy_k^*\\
&=\sum_{j=1}^{n}\sum_{k=1}^{n}(\cos z_j \cos z_k+ \sin z_j \sin z_k)y_jy_k^*\\
&=\left(\sum_{j=1}^{n} (\cos z_j)y_j\right)\left(\sum_{j=1}^{n} (\cos z_j)y_j\right)^*+\left(\sum_{j=1}^{n} (\sin z_j)y_j\right)\left(\sum_{j=1}^{n} (\sin z_j)y_j\right)^*\\
&\geq 0.
\end{align*}
We define $ n$ by $n $ matrices $M_1,\dots, M_d$ as follows. 
\begin{align*}
M_l \coloneqq \begin{bmatrix}\cos \left(\frac{x_{j,l}-x_{k,l}}{2}\right)\end{bmatrix}_{1\leq j,k \leq n}, \quad \forall 1\leq l \leq d.
\end{align*}
Theorem \ref{COMMUTATIVESCHURPRODUCT} then says that the matrix 
\begin{align*}
M\coloneqq M_1\circ \cdots \circ  M_d=\begin{bmatrix}\prod_{l=1}^{d}\cos\left(\frac{x_{j,l}-x_{k,l}}{2}\right) \end{bmatrix}_{1\leq j,k \leq n} 
\end{align*}
is positive. Since all diagonal entries of $M$ are one's, we can apply Corollary \ref{SECONDCOROLLARY} to get 

\begin{align*}
\begin{bmatrix} \prod_{l=1}^{d}\frac{1+\cos (x_{j,l}-x_{k,l})}{2}\end{bmatrix}_{1\leq j,k \leq n}=\begin{bmatrix} \prod_{l=1}^{d}\cos^2 \left(\frac{x_{j,l}-x_{k,l}}{2}\right)\end{bmatrix}_{1\leq j,k \leq n}=M\circ M \succeq \frac{1}{n}E_n,
\end{align*}
i.e., 
\begin{align*}
\begin{bmatrix}\prod_{l=1}^{d}\frac{1+\cos (x_{j,l}-x_{k,l})}{2}-\frac{1}{n}\end{bmatrix}_{1\leq j,k \leq n}\succeq 0.
\end{align*}
\end{proof}
We end the paper by asking an open problem  similar to question asked by Vyb\'{\i}ral in arXiv version  (see https://arxiv.org/abs/1909.11726v1) of the paper \cite{VYBIRAL}.
\begin{question}\label{LASTQUESTION}
\textbf {Can the bound in Theorem \ref{MYTHEOREM}  be improved for the C*-algebraic Schur product of positive matrices over (commutative) unital C*-algebras?}
\end{question}
\textbf{Final sentense:}
  Improved version  of Theorem   \ref{VIBIRALSECOND}  is given by  Dr. Apoorva Khare (see Theorem A in \cite{APOORVAPROC})  but it seems that the arguments used in the proof of Theorem A in \cite{APOORVAPROC} do not work for C*-algebras. 
  \section{Acknowledgements}
 I thank Dr. P. Sam Johnson, Department of Mathematical and Computational Sciences, National Institute of Technology Karnataka (NITK),  Surathkal for his help and some discussions.
 \bibliographystyle{plain}
 \bibliography{reference.bib}

\end{document}